\documentclass[11pt,reqno]{amsart} 
\title{Generalized Analogs of the Heisenberg Uncertainty Inequality}
\usepackage{amsmath,amssymb,amsthm,amsrefs,mathrsfs}
\numberwithin{equation}{section}
\newtheorem{thm}{\sc Theorem}[section]

\newtheorem{prop}[thm]{\sc Proposition}

\newtheorem{ex}[thm]{\sc Example}
\newcommand{\R}{\mathbb{R}}

\newcommand{\h}{\mathbb{H}}

\newcommand{\C}{\mathbb{C}}
\newcommand{\F}{\mathscr{F}}
\newcommand{\CS}{\mathcal{S}}
\newcommand{\g}{\mathfrak{g}}

\newcommand{\hf}{\mathfrak{h}}
\newcommand{\z}{\mathfrak{z}}
\newcommand{\OO}{\mathcal{O}}
\newcommand{\U}{\mathcal{U}}
\newcommand{\W}{\mathcal{W}}
\newcommand{\B}{\mathcal{B}}

\DeclareMathOperator*{\expo}{exp}

\DeclareMathOperator*{\pf}{Pf}
\DeclareMathOperator*{\dint}{\displaystyle\int}
\allowdisplaybreaks
\begin{document}
\author{ASHISH BANSAL}
\address{\noindent Department of Mathematics, Keshav Mahavidyalaya (University of Delhi), H-4-5 Zone, Pitampura, Delhi 110034, India}
\email{abansal@keshav.du.ac.in}
\author{AJAY KUMAR}
\address{Department of Mathematics, University of Delhi, Delhi 110007, India}
\email[Corresponding author]{akumar@maths.du.ac.in}
\thanks{}
\begin{abstract}
We investigate locally compact topological groups for which a generalized analogue of Heisenberg uncertainty inequality hold. In particular, it is shown that this inequality holds for $\R^n \times K$ (where $K$ is a separable unimodular locally compact group of type I), Euclidean Motion group and several general classes of nilpotent Lie groups which include thread-like nilpotent Lie groups, $2$-NPC nilpotent Lie groups and several low-dimensional nilpotent Lie groups. 
\end{abstract}
\maketitle
\section{Introduction}
In $1927$, Werner Heisenberg gave a principle related to the uncertainties in the measurements of position and momentum of microscopic particles. This principle is known as \textit{Heisenberg uncertainty principle} and can be stated as follows:
\begin{quotation}
\textit{It is impossible to know simultaneously the exact position and momentum of a particle. That is, the more exactly the position is determined, the less known the momentum, and vice versa.}
\end{quotation}

In 1933, N. Wiener gave the following mathematical formulation of the Heisenberg uncertainty principle:
\begin{quotation}
\textit{A nonzero function and its Fourier transform cannot both be sharply localized.}
\end{quotation}

The \textit{Heisenberg's uncertainty inequality} is precise quantitative formulation of the above principle.  

The Fourier transform of $f \in L^1(\R^n)$ is given by,
\begin{flalign*}
&&\widehat{f}(\xi)&=\dint\limits_{\R^n}{f(x)\ e^{-2\pi i \langle{x,\xi}\rangle}}\ dx,&
\end{flalign*}
where $\langle{\cdot ,\cdot }\rangle$ denotes the usual inner product on $\R^n$. This definition of Fourier transform holds for functions in $L^1(\R^n) \cap L^2(\R^n)$. Since, $L^1(\R^n) \cap L^2(\R^n)$ is dense in $L^2(\R^n)$, the definition of Fourier transform can be extended to the functions in $L^2(\R^n)$. 

The following theorem gives the Heisenberg uncertainty inequality for the Fourier transform on $\R^n$. For proof of the theorem, see \cite{Fol:Sit:97}.
\begin{thm}\label{theq-Rn}
For any $f \in L^2(\R^n)$, we have
\begin{flalign}
&&\dfrac{n\|f\|_2^2}{4\pi}&\leq \left(\dint\limits_{\R^n}{\|x\|^2\ |f(x)|^2}\ dx\right)^{1/2}\left(\dint\limits_{\R^n}{\|y\|^2\ |\widehat{f}(y)|^2}\ dy\right)^{1/2},   \label{heq-Rn}&
\end{flalign}
where $\|\cdot\|_2$ denotes the $L^2$-norm and $\|\cdot\|$ denotes the Euclidean norm. 
\end{thm}

The Heisenberg uncertainty inequality has been established for the Fourier transform on the Heisenberg group by Thangavelu \cite{Tha:90}. Further generalizations of the inequality on the Heisenberg group have been establishd by Sitaram et al.  \cite{Sit:Sun:Tha:95} and Xiao et al. \cite{Xia:He:12}. For some more details, see \cite{Fol:Sit:97}.

The inequality given below can be proved using H\"older's inequality and the inequality \eqref{heq-Rn}.
\begin{thm}
For any $f \in L^2(\R^n)$ and $a,b \geq 1$, we have
\begin{flalign*}
&&\dfrac{n\|f\|_2^{\left(\frac{1}{a}+\frac{1}{b}\right)}}{4\pi} &\leq \left(\dint_{\R^n}{\|x\|^{2a}\ |f(x)|^2}\ dx\right)^{\frac{1}{2a}}\left(\dint_{\R^n}{\|y\|^{2b}\ |\widehat{f}(y)|^2}\ dy\right)^{\frac{1}{2b}},&
\end{flalign*}
where $\|\cdot\|_2$ denotes the $L^2$-norm and $\|\cdot\|$ denotes the Euclidean norm. 
\end{thm}

In section $2$, we shall prove a generalized analogue of Heisenberg uncertainty inequality for $\R^n \times K$, where $K$ is a separable unimodular locally compact group of type I. In the next section, a generalized analogue of Heisenberg uncertainty inequality for Euclidean motion group $M(n)$ is proved. The last section deals with a generalized analogue of Heisenberg uncertainty inequality for several general classes of nilpotent Lie groups for which the Hilbert-Schmidt norm of the group Fourier transform $\pi_\xi(f)$ of $f$ attains a particular form. Theses classes include thread-like nilpotent Lie groups, $2$-NPC nilpotent Lie groups and several low-dimensional nilpotent Lie groups. 
\section{$\R^n \times K$, $K$ a locally compact group}
Consider $G=\R^n \times K$, where $K$ is a separable unimodular locally compact group of type I. The Haar measure of $G$ is $dg=dx\ dk$, where $dx$ is Lebesgue measure on $\R^n$ and $dk$ is the left Haar measure on $K$. The dual $\widehat{G}$ of $G$ is $\R^n \times \widehat{K}$, where $\widehat{K}$ is the dual space of $K$. 

The Fourier transform of $f \in L^2(G)$ is given by,
\begin{flalign*}
&& \widehat{f}(y,\sigma)&=\dint_{\R^n}\dint_{K}{f(x,k)\ e^{-2\pi i \langle{x,y}\rangle}\ \sigma(k^{-1})}\ dk\ dx, &
\end{flalign*} 
for $(y,\sigma) \in \R^n \times \widehat{K}$.  

\begin{thm} \label{theq-RnK}
For any $f \in L^2(\R^n \times K)$ (where $K$ is a separable unimodular locally compact group of type I) and $a,b \geq 1$, we have
\begin{flalign}
\noindent \dfrac{n\|f\|_2^{\left(\frac{1}{a}+\frac{1}{b}\right)}}{4\pi} &\leq \left(\dint_{\R^n}\dint_{K}{\|x\|^{2a}\ |f(x,k)|^2}\ dk\ dx\right)^{\frac{1}{2a}}\left(\dint_{\R^n}\dint_{\widehat{K}}{\|y\|^{2b}\ \|\widehat{f}(y,\sigma)\|_{\text{HS}}^2}\ dy\ d\sigma\right)^{\frac{1}{2b}}. &\label{heq-RnK}
\end{flalign}
\end{thm}
\begin{proof}
Without loss of generality, we may assume that both the integrals on right hand side of \eqref{heq-RnK} are finite.\\
Given that $f \in L^2(\R^n \times K)$, there exists $A \subseteq K$ of measure zero such that for $k \in K\setminus A=A'$ (say), we have
\begin{flalign*}
&&\dint\limits_{\R^n}{|f(x,k)|^2}\ dx &< \infty.&
\end{flalign*}
For all $k \in A'$, we define $f_k(x)=f(x,k)$, for every $x\in \R^n$. \\
Clearly, for all $k\in A'$, $f_k \in L^2(\R^n)$ and for all $y \in \R^n$,
\begin{flalign*}
&&\widehat{f_k}(y)&=\dint\limits_{\R^n}{f(x,k)\ e^{-2\pi i \langle{x,y}\rangle}}\ dy=\F_1 f(y,k).&
\end{flalign*}
By Theorem \ref{theq-Rn}, we have
\begin{flalign*}
&&\dfrac{n}{4\pi}\dint_{\R^n}{|f(x,k)|^2}\ dx&\leq\left(\dint_{\R^n}{\|x\|^2\ |f_k(x)|^2}\ dx\right)^{1/2}\left(\dint_{\R^n}{\|y\|^2\ |\widehat{f_k}(y)|^2}\ dy\right)^{1/2}. & 
\end{flalign*}
Integrating both sides with respect to $dk$, we obtain
\begin{flalign*}
&&\dfrac{n}{4\pi}\dint\limits_{A'}\dint\limits_{\R^n}{|f(x,k)|^2}\ dx\ dk&\leq \dint_{A'}\left(\dint_{\R^n}{\|x\|^2\ |f_k(x)|^2}\ dx\right)^{1/2}\times &\\
&&&\qquad \left(\dint_{\R^n}{\|y\|^2\ |\widehat{f_k}(y)|^2}\ dy\right)^{1/2}\ dk.&  
\end{flalign*}
The integral on the L.H.S. is equal to $\|f\|_2^2$, so using Cauchy Schwarz inequality and Fubini's theorem, we have
\begin{flalign}
&&\dfrac{n\|f\|_2^2}{4\pi}&\leq \left(\dint_{K}\dint_{\R^n}{\|x\|^2\ |f(x,k)|^2}\ dx\ dk\right)^{1/2}\left(\dint_{\R^n}\|y\|^2\dint_{A'}{|\widehat{f_k}(y)|^2}\ dk\ dy\right)^{1/2}. \label{step1-RnK}&
\end{flalign}
Now, using H\"older's inequality, we have
\begin{flalign*}
&\left(\dint_{\R^n}\dint_{K}{\|x\|^{2a}\ |f(x,k)|^2}\ dk\ dx\right)^{\frac{1}{a}}\left(\dint_{\R^n}\dint_{K}{|f(x,k)|^2}\ dk\ dx\right)^{1-\frac{1}{a}} &\\
&\geq \dint_{\R^n}\dint_{K}{\|x\|^2\ |f(x,k)|^{\frac{2}{a}}|f(x,k)|^{2\left(1-\frac{1}{a}\right)}}\ dk\ dx  &\\
&= \dint_{\R^n}\dint_{K}{\|x\|^2\ |f(x,k)|^2}\ dk\ dx,  &
\end{flalign*}
which implies
\begin{flalign}
\dint_{\R^n}\dint_{K}{\|x\|^2\ |f(x,k)|^2}\ dk\ dx &\leq \left(\dint_{\R^n}\dint_{K}{\|x\|^{2a}\ |f(x,k)|^2}\ dk\ dx\right)^{\frac{1}{a}}\left(\|f\|_2^2\right)^{1-\frac{1}{a}}.  \label{step2-RnK}&
\end{flalign}
Combining \eqref{step1-RnK} and \eqref{step2-RnK}, we obtain
\begin{flalign}
\dfrac{n\|f\|_2^2}{4\pi}&\leq \left(\dint_{\R^n}\dint_{K}{\|x\|^{2a}\ |f(x,k)|^2}\ dk\ dx\right)^{\frac{1}{2a}}\left(\|f\|_2^2\right)^{\frac{1}{2}-\frac{1}{2a}} \left(\dint_{\R^n}\|y\|^2\dint_{A'}{|\widehat{f_k}(y)|^2}\ dk\ dy\right)^{1/2}. \label{step3-RnK}&
\end{flalign}
\begin{flalign*}
\text{Since,}\ \ \dint_{\R^n}\dint_{A'}{|\F_1f(y,k)|^2}\ dy\ dk &=\dint_{\R^n}\dint_{A'}{|f(x,k)|^2}\ dx\ dk =\|f\|_2^2 < \infty, &
\end{flalign*}
therefore, $\F_1f \in L^2(\R^n\times A')$. So, $\F_2\F_1 f$ is well defined a.e. By approximating $f\in L^2(\R^n \times A')$ by functions in $L^1\cap L^2(\R^n\times A')$, we have 
\begin{flalign*}
&&\F_2\F_1 f&=\widehat{f},&
\end{flalign*}
for all $f \in L^2(\R^n \times A')$. Applying Plancherel formula on the locally compact group $K$, we have
\begin{flalign*}
&&\dint_{A'}{|\widehat{f_k}(y)|^2}\ dk &=\dint_{\widehat{K}}{\|\widehat{f}(y,\sigma)\|_{\text{HS}}^2}\ d\sigma. &
\end{flalign*}
Thus, \eqref{step3-RnK} can be written as
\begin{flalign}
\dfrac{n\|f\|_2^2}{4\pi}&\leq \left(\dint_{\R^n}\dint_{K}{\|x\|^{2a}\ |f(x,k)|^2}\ dk\ dx\right)^{\frac{1}{2a}}\left(\|f\|_2^2\right)^{\frac{1}{2}-\frac{1}{2a}} \left(\dint_{\R^n}\dint_{\widehat{K}}{\|y\|^2\|\widehat{f}(y,\sigma)\|_{\text{HS}}^2}\ dy\ d\sigma\right)^{1/2}.  \label{step4-RnK}&
\end{flalign}
Now, again using H\"older's inequality, we have
\begin{flalign*}
&\left(\dint_{\R^n}\dint_{\widehat{K}}{\|y\|^{2b}\ \|\widehat{f}(y,\sigma)\|_{\text{HS}}^2}\ dy\ d\sigma\right)^{\frac{1}{b}}\left(\dint_{\R^n}\dint_{\widehat{K}}{\|\widehat{f}(y,\sigma)\|_{\text{HS}}^2}\ dy\ d\sigma\right)^{1-\frac{1}{b}}&\\
&\geq \dint_{\R^n}\dint_{\widehat{K}}{\|y\|^2\ \|\widehat{f}(y,\sigma)\|_{\text{HS}}^{\frac{2}{b}}\ \|\widehat{f}(y,\sigma)\|_{\text{HS}}^{2\left(1-\frac{1}{b}\right)}}\ dy\ d\sigma & \\
&= \dint_{\R^n}\dint_{\widehat{K}}{\|y\|^2\ \|\widehat{f}(y,\sigma)\|_{\text{HS}}^2}\ dy\ d\sigma, &
\end{flalign*}
which implies
\begin{flalign}
&\dint_{\R^n}\dint_{\widehat{K}}{\|y\|^2\ \|\widehat{f}(y,\sigma)\|_{\text{HS}}^2}\ dy\ d\sigma\leq \left(\dint_{\R^n}\dint_{\widehat{K}}{\|y\|^{2b}\ \|\widehat{f}(y,\sigma)\|_{\text{HS}}^2}\ dy\ d\sigma\right)^{\frac{1}{b}}\left(\|f\|_2^2\right)^{1-\frac{1}{b}}.\label{step5-RnK}&
\end{flalign}
Combining \eqref{step4-RnK} and \eqref{step5-RnK}, we obtain
\begin{flalign*}
&&\dfrac{n\|f\|_2^2}{4\pi}&\leq \left(\dint_{\R^n}\dint_{K}{\|x\|^{2a}\ |f(x,k)|^2}\ dk\ dx\right)^{\frac{1}{2a}}\left(\|f\|_2^2\right)^{\frac{1}{2}-\frac{1}{2a}}\times &\\*
&&&\qquad \left(\dint_{\R^n}\dint_{\widehat{K}}{\|y\|^{2b}\ \|\widehat{f}(y,\sigma)\|_{\text{HS}}^2}\ dy\ d\sigma\right)^{\frac{1}{2b}}\left(\|f\|_2^2\right)^{\frac{1}{2}-\frac{1}{2b}}, &
\end{flalign*}
which implies
\begin{flalign*}
\dfrac{n\|f\|_2^{\left(\frac{1}{a}+\frac{1}{b}\right)}}{4\pi} &\leq \left(\dint_{\R^n}\dint_{K}{\|x\|^{2a}\ |f(x,k)|^2}\ dk\ dx\right)^{\frac{1}{2a}}\left(\dint_{\R^n}\dint_{\widehat{K}}{\|y\|^{2b}\ \|\widehat{f}(y,\sigma)\|_{\text{HS}}^2}\ dy\ d\sigma\right)^{\frac{1}{2b}}.&
\end{flalign*}
\end{proof}
\section{Euclidean Motion Group $M(n)$}
Consider $M(n)$ to be the semi-direct product of $\R^n$ with $K=SO(n)$. The group law is given by,
\begin{flalign*}
&&(z,k)(w,k')&=(z+k\cdot w, kk'),&
\end{flalign*}
for $z,w \in \R^n$ and $k,k' \in K$. The group $M(n)$ is called the \textit{Motion Group} of the Euclidean plane $\R^n$. 

As in \cite{Sar:Tha:04}, $M=SO(n-1)$ can be considered as a subgroup of $K$ leaving the point $e_1=(1,0,0,\ldots,0)$ fixed. All the irreducible unitary representations of $M(n)$ relevant for the Plancherel formula are parametrized (upto unitary equivalence) by pairs $(\lambda,\sigma)$, where $\lambda >0$ and $\sigma \in \widehat{M}$, the unitary dual of $M$.

Given $\sigma \in \widehat{M}$ realized on a Hilbert space $H_{\sigma}$ of dimension $d_\sigma$, consider the space,
\begin{flalign*}
L^2(K,\sigma)&=\left\{\varphi\ |\ \varphi: K \rightarrow M_{d_\sigma \times d_\sigma}, \dint{\|\varphi(k)\|^2}\ dk < \infty, \varphi(uk)=\sigma(u)\varphi(k),\ \text{for}\ u \in M\ \text{and}\ k\in K\right\}.&
\end{flalign*}
Note that $L^2(K,\sigma)$ is a Hilbert space under the inner product
\begin{flalign*}
&&\langle{\varphi,\psi}\rangle&=\dint_{K}{tr(\varphi(k)\psi(k)^\ast)}\ dk.&
\end{flalign*}
For each $\lambda >0$ and $\sigma \in \widehat{M}$, we can define a representation $\pi_{\lambda,\sigma}$ of $M(n)$ on $L^2(K,\sigma)$ as follows:

For $\varphi \in L^2(K,\sigma)$, $(z,k)\in M(n)$,
\begin{flalign*}
&&\pi_{\lambda,\sigma}(z,k)\varphi(u)&=e^{i\lambda\langle{u^{-1}\cdot e_1,z\rangle}}\ \varphi(uk),&
\end{flalign*}
for $u\in K$. 

If $\varphi_j(k)$ are the column vectors of $\varphi \in L^2(K,\sigma)$, then $\varphi_j(uk)=\sigma(u)\varphi_j(k)$ for all $u \in M$. Therefore, $L^2(K,\sigma)$ can be written as the direct sum of $d_\sigma$ copies of $H(K,\sigma)$, where 
\begin{flalign*}
H(K,\sigma)&=\left\{\varphi\ |\ \varphi: K \rightarrow \C^{d_\sigma}, \dint{\|\varphi(k)\|^2}\ dk < \infty, \varphi(uk)=\sigma(u)\varphi(k),\ \text{for}\ u \in M\ \text{and}\ k\in K\right\}.&
\end{flalign*}
It can be shown that $\pi_{\lambda,\sigma}$ restricted to $H(K,\sigma)$ is an irreducible unitary representation of $M(n)$. Moreover, any irreducible unitary representation of $M(n)$ which is infinite dimensional is unitarily equivalent to one and only one $\pi_{\lambda,\sigma}$.  \\ 
The Fourier transform of $f \in L^2(M(n))$ is given by,
\begin{flalign*}
&& \widehat{f}(\lambda,\sigma)&=\dint_{M(n)}{f(z,k)\ \pi_{\lambda,\sigma}(z,k)^\ast}\ dz\ dk. &
\end{flalign*} 
$\widehat{f}(\lambda,\sigma)$ is a Hilbert-Schmidt operator on $H(K,\sigma)$. 

A solid harmonic of degree $m$ is a polynomial which is homogeneous of degree $m$ and whose Laplacian is zero. The set of all such polynomials will be denoted by $\h_m$ and the restrictions of elements of $\h_m$ to $S^{n-1}$ is denoted by $S_m$. By choosing an orthonormal basis $\{g_{mj} : j = 1, 2,\ldots,d_m\}$ of $S_m$ for each $m = 0, 1, 2, \ldots$, we get an orthonormal basis for $L^2(S^{n-1})$. 

The Haar measure on $M(n)$ is $dg=dz\ dk$, where $dz$ is Lebesgue measure on $\R^n$ and $dk$ is the normalized Haar measure on $SO(n)$. 

The Plancherel formula on $M(n)$ is given as follows (See \cite{Kum:Oka:73}):
\begin{prop}[Plancherel Formula]
Let $f \in L^2(M(n))$, then \label{PF-Mn}
\begin{flalign*}
\dint_{M(n)}{\left|f(z_1,z_2,\ldots,z_n,k)\right|^2}\ dz_1\ dz_2\ \ldots\ \ dz_n\ dk&=c_n\dint_{0}^{\infty}\left(\sum_{\sigma \in \widehat{M}}{d_\sigma\|\widehat{f}(\lambda,\sigma)\|_{\text{HS}}^2}\right)\ \lambda^{n-1}\ d\lambda &
\end{flalign*}
where $c_n=\dfrac{2}{2^{n/2}\ \Gamma\left(\frac{n}{2}\right)}$.
\end{prop}

We shall now state and prove the following generalized Heisenberg uncertainty inequality for Fourier transform on $M(n)$:
\begin{thm}\label{theq-Mn}
For any $f \in L^2(M(n))$ and $a,b \geq 1$, we have
\begin{flalign}
\dfrac{\|f\|_2^{\left(\frac{1}{a}+\frac{1}{b}\right)}}{2\sqrt{c_n}}&\leq\left(\dint_{K}\dint_{\R^n}{\|z\|^{2a}\ |f(z,k)|^2}\ dz\ dk\right)^{\frac{1}{2a}}\left(\dint_{0}^{\infty}\sum_{\sigma\in \widehat{M}}{d_\sigma\ \lambda^{2b}\|\widehat{f}(\lambda,\sigma)\|_{\text{HS}}^2}\ \lambda^{n-1}\ d\lambda\right)^{\frac{1}{2b}}. \label{heq-Mn}&
\end{flalign}
\end{thm}
\begin{proof}
Consider the norm $\|\cdot\|$ on $L^2(M(n))$ defined by
\begin{flalign*}
\|f\|:&=\left(\dint_{\R^n}\dint_{K}{(1+\|z\|^{2a})\ |f(z,k)|^2}\ dz\ dk\right)^{1/2}+\left(\dint_{0}^{\infty}\sum_{\sigma\in \widehat{M}}{d_\sigma (1+\lambda^{2b})\|\widehat{f}(\lambda,\sigma)\|_{\text{HS}}^2}\ \lambda^{n-1}\ d\lambda\right)^{1/2}.&
\end{flalign*}
This gives us a Banach space $B=\{f\in L^2(G):\|f\|<\infty\}$, which is contained in $L^2(M(n))$ and the space $\CS(M(n))$ of $C^\infty$-functions which are rapidly decreasing on $M(n)$ can be shown to be dense in $B$. It suffices to prove the inequality of Theorem \ref{theq-Mn} for functions in $\CS(M(n))$, it is automatically valid for any $f\in B$. If $0\neq f \in L^2(M(n))\setminus B$, then the right hand side of the inequality is always $+\infty$ and the inequality is trivially valid.\\
Let $f\in \CS(M(n))$. Assuming that both the integrals on right hand side of \eqref{heq-Mn} are finite, we have
\begin{flalign*}
&&\dint\limits_{\R^n}{|f(z,k)|^2}\ dz &< \infty,\ \text{for all}\ k\in K.&
\end{flalign*}
For $k \in K$, we define $f_k(z)=f(z,k)$, for every $z \in\R^n$. \\
Clearly, $f_k \in L^2(\R^n)$, for all $k\in K$. \\
Take $z=(z_1,z_2,\ldots,z_n)$ and $w=(w_1,w_2,\ldots,w_n)$. \\
By Heisenberg inequality on $\R^n$, we have
\begin{flalign*}
&&\dfrac{\|f_k\|_2^2}{4\pi}&\leq\left(\dint_{\R^n}{|z_1|^2\ |f_k(z)|^2}\ dz\right)^{1/2}\left(\dint_{\R^n}{|w_1|^2\ |\widehat{f_k}(w)|^2}\ dw\right)^{1/2} & \\
&\Rightarrow &\dfrac{1}{4\pi}\dint_{\R^n}{|f(z,k)|^2}\ dz &\leq\left(\dint_{\R^n}{|z_1|^2\ |f(z,k)|^2}\ dz\right)^{1/2}\left(\dint_{\R^n}{|w_1|^2\ |\widehat{f_k}(w)|^2}\ dw\right)^{1/2}  & 
\end{flalign*}
Integrating both sides with respect to $dk$, we get
\begin{flalign*}
\dfrac{1}{4\pi}\dint_{K}\dint_{\R^n}{|f(z,k)|^2}\ dz\ dk &\leq\dint_{K}\left(\dint_{\R^n}{|z_1|^2\ |f(z,k)|^2}\ dz\right)^{1/2}\left(\dint_{\R^n}{|w_1|^2\ |\widehat{f_k}(w)|^2}\ dw\right)^{1/2}\ dk & 
\end{flalign*}
which implies
\begin{flalign}
\dfrac{\|f\|_2^2}{4\pi}&\leq\dint_{K}\left(\dint_{\R^n}{|z_1|^2\ |f(z,k)|^2}\ dz\right)^{1/2}\left(\dint_{\R^n}{|w_1|^2\ |\widehat{f_k}(w)|^2}\ dw\right)^{1/2}\ dk \nonumber&\\
&\leq\left(\dint_{K}\dint_{\R^n}{|z_1|^2\ |f(z,k)|^2}\ dz\ dk\right)^{1/2}\left(\dint_{K}\dint_{\R^n}{|w_1|^2\ |\widehat{f_k}(w)|^2}\ dw\ dk\right)^{1/2} \nonumber\hspace{-20pt}&\\
&\hspace{190pt}[\text{By Cauchy Schwarz Inequality}] \nonumber&\\
&\leq\left(\dint_K\dint_{\R^n}{\|z\|^2\ |f(z,k)|^2}\ dz\ dk\right)^{1/2}\left(\dint_{K}\dint_{\R^n}{|w_1|^2\ |\widehat{f_k}(w)|^2}\ dw\ dk\right)^{1/2}. \label{step01-Mn}\hspace{-20pt}&
\end{flalign}
Now,
\begin{flalign}
&\left(\dint_{K}\dint_{\R^n}{\|z\|^{2a}\ |f(z,k)|^2}\ dz\ dk\right)^{\frac{1}{a}}\left(\dint_{K}\dint_{\R^n}{|f(z,k)|^2}\ dz\ dk\right)^{1-\frac{1}{a}}\nonumber& \\
&=\left(\dint_{K}\dint_{\R^n}{\left(\|z\|^2\ |f(z,k)|^{\frac{2}{a}}\right)^{a}}\ dz\ dk\right)^{\frac{1}{a}}\left(\dint_{K}\dint_{\R^n}{\left(|f(z,k)|^{2\left(1-\frac{1}{a}\right)}\right)^{\frac{1}{\left(1-\frac{1}{a}\right)}}}\ dz\ dk\right)^{1-\frac{1}{a}}\nonumber& \\
&\geq \dint_{K}\dint_{\R^n}{\|z\|^2\ |f(z,k)|^{\frac{2}{a}}|f(z,k)|^{2\left(1-\frac{1}{a}\right)}}\ dz\ dk \hspace{65pt} [\text{By Holder's Inequality}]  \nonumber&\\
&= \dint_{K}\dint_{\R^n}{\|z\|^2\ |f(z,k)|^2}\ dz\ dk   \label{step02-Mn}&
\end{flalign}
Combining \eqref{step01-Mn} and \eqref{step02-Mn}, we get
\begin{flalign}
\dfrac{\|f\|_2^2}{4\pi}&\leq\left(\dint_{K}\dint_{\R^n}{\|z\|^{2a}\ |f(z,k)|^2}\ dz\ dk\right)^{\frac{1}{2a}}\left(\|f\|_2^2\right)^{\frac{1}{2}-\frac{1}{2a}}\nonumber &\\
&\qquad \times\left(\dint_{K}\dint_{\R^n}{|w_1|^2\ |\widehat{f_k}(w)|^2}\ dw\ dk\right)^{1/2} \label{step3-Mn}&
\end{flalign}
Now, using Plancherel formula on $\R^n$, we have
\begin{flalign}
&\dint_{K}\dint_{\R^n}{|w_1|^2\ |\widehat{f_k}(w)|^2}\ dw\ dk \nonumber &\\
&=\dint_K\dint_{\R^n}{|w_1|^2\ \left|\dint_{\R^n}{f(z,k)\ e^{-2\pi i\langle{z,w}\rangle}}\ dz\right|^2}\ dw\ dk\nonumber &\\
&=\dint_K\dint_{\R^n}{|w_1|^2\ |\F_{1,2,\ldots, n} f(w_1,w_2,\ldots,w_n,k)|^2}\ dw_1\ dw_2\ \ldots\ dw_n\ dk \nonumber &\\
&=\dint_K\dint_{\R^n}{|w_1|^2\ |\F_{1} f(w_1,z_2,\ldots,z_n,k)|^2}\ dw_1\ dz_2\ \ldots\ dz_n\ dk. \label{step4-Mn}&
\end{flalign}
Since, $\dfrac{\partial f}{\partial z_1}\in \CS(M(n))$, we have 
\begin{flalign*}
&&\dint_{\R}{\left|\dfrac{\partial f}{\partial z_1}(z_1,z_2,\ldots,z_n,k)\right|^2}\ dz_1&< \infty,&
\end{flalign*}
for all $z_i \in \R$ and $k \in K$.\\
So, $w_1\F_1 f(w_1,z_2,\ldots,z_n,k)\in L^2(\R)$ and
\begin{flalign*}
&&\left(\dfrac{\partial f}{\partial z_1}(z_1,z_2,\ldots,z_n,k)\right)^{\widehat{\ }}(w_1)&=2\pi i w_1\F_1 f(w_1,z_2,\ldots,z_n,k).&
\end{flalign*}
for all $z_i \in \R$ and $k \in K$. Then,
\begin{flalign*}
&&\dint_{\R}{|w_1|^2\ |\F_1 f(w_1,z_2,\ldots,z_n,k)|^2}\ dw_1&=\dfrac{1}{4\pi^2}\dint_{\R}{\left|\dfrac{\partial f}{\partial z_1}(z_1,z_2,\ldots,z_n,k)\right|^2}\ dz_1, &
\end{flalign*}
which implies
\begin{flalign}
&\dint_{K}\dint_{\R^n}{|w_1|^2\ |\F_1 f(w_1,z_2,\ldots,z_n,k)|^2}\ dw_1\ dz_2\ \ldots\ dz_n\ dk \nonumber&\\
&=\dfrac{1}{4\pi^2}\dint_{K}\dint_{\R^n}{\left|\dfrac{\partial f}{\partial z_1}(z_1,z_2,\ldots,z_n,k)\right|^2}\ dz_1\ dz_2\ \ldots\ dz_n\ dk. \label{step5-Mn}&
\end{flalign}
By Proposition \ref{PF-Mn}, we obtain
\begin{flalign}
&\dint_K\dint_{\R^n}{\left|\dfrac{\partial f}{\partial z_1}(z_1,z_2,\ldots,z_n,k)\right|^2}\ dz_1\ dz_2\ \ldots\ dz_n\ dk \nonumber&\\*
&=c_n\dint_{0}^{\infty}\sum_{\sigma\in \widehat{M}}{d_\sigma\left\|\left(\dfrac{\partial f}{\partial z_1}\right)^{\widehat{\ }}(\lambda,\sigma)\right\|_{\text{HS}}^2}\ \lambda^{n-1}\ d\lambda. \label{step6-Mn}
\end{flalign}
Combining \eqref{step3-Mn}, \eqref{step4-Mn}, \eqref{step5-Mn} and \eqref{step6-Mn}, we obtain
\begin{flalign}
&&\dfrac{\|f\|_2^2}{2\sqrt{c_n}}&\leq\left(\dint_{K}\dint_{\R^n}{\|z\|^{2a}\ |f(z,k)|^2}\ dz\ dk\right)^{\frac{1}{2a}}\left(\|f\|_2^2\right)^{\frac{1}{2}-\frac{1}{2a}}\times\nonumber &\\
&&&\qquad \left(\dint_{0}^{\infty}\sum_{\sigma\in \widehat{M}}{d_\sigma\ \left\|\left(\dfrac{\partial f}{\partial z_1}\right)^{\widehat{\ }}(\lambda,\sigma)\right\|_{\text{HS}}^2}\ \lambda^{n-1}\ d\lambda\right)^{1/2}. \label{step7-Mn} &
\end{flalign}
For each $\lambda >0$ and $\sigma \in \widehat{M}$, consider the representation $\pi_{\lambda,\sigma}(z,k)$ realised on $L^2(K,\sigma)$ as,
\begin{flalign*}
&&\pi_{\lambda,\sigma}(z,k)g(u)&=e^{i\lambda\langle{u^{-1}\cdot e_1,z}\rangle}\ g(uk),\quad u \in SO(n).&
\end{flalign*}
Denote $u=[u_{ij}]_{n\times n}$, we have 
\begin{flalign*}
&&u^{-1}\cdot e_1&=u^{T}\cdot e_1=[u_{11}\ u_{12}\ \ldots\ u_{1n}]^T.&
\end{flalign*}
So, $\langle{u^{-1}\cdot e_1,z}\rangle=\sum\limits_{i=1}^{n}{u_{1i}z_i}$.\\ 
Since, $f\in\CS(M(n))$,
\begin{flalign*}
&\left(\dfrac{\partial f}{\partial z_1}\right)^{\widehat{\ }}(\lambda,\sigma)g(u)&\\[4pt]
&=\dint_{\R^n}\dint_{K}{\dfrac{\partial f}{\partial z_1}(z_1,z_2,\ldots,z_n,k)\ \pi_{\lambda,\sigma}(z_1,z_2,\ldots,z_n,k)^\ast g(u)}\ dz_1\ dz_2\ \ldots\ dz_n\ dk &\\[4pt]
&=\dint_{\R^n}\dint_{K}{\lim_{h\rightarrow 0}{\left[\dfrac{f(z_1+h,z_2,\ldots,z_n,k)-f(z_1,z_2,\ldots,z_n,k)}{h}\right]}\ \pi_{\lambda,\sigma}(z_1,z_2,\ldots,z_n,k)^\ast g(u)}&\\*
&\hspace{200pt}\ dz_1\ dz_2\ \ldots\ dz_n\ dk &\\
&=\lim_{h\rightarrow 0}\dfrac{1}{h}\left[\dint_{\R^n}\dint_{K}{f(z_1+h,z_2,\ldots,z_n,k)\ \pi_{\lambda,\sigma}(z_1,z_2,\ldots,z_n,k)^\ast g(u)}\ dz_1\ dz_2\ \ldots\ dz_n\ dk\right.&\\[4pt]
&\qquad \left.-\dint_{\R^n}\dint_{K}{f(z_1,z_2,\ldots,z_n,k)\ \pi_{\lambda,\sigma}(z_1,z_2,\ldots,z_n,k)^\ast g(u)}\ dz_1\ dz_2\ \ldots\ dz_n\ dk\right] &\\[4pt]
&=\lim_{h\rightarrow 0}\dfrac{1}{h}\left[\dint_{\R^n}\dint_{K}{f(z_1,z_2,\ldots,z_n,k)\ e^{-i \lambda h u_{11}}\pi_{\lambda,\sigma}(z_1,z_2,\ldots,z_n,k)^\ast g(u)}\ dz_1\ dz_2\ \ldots\ dz_n\ dk\right.&\\
&\qquad \left.-\dint_{\R^n}\dint_{K}{f(z_1,z_2,\ldots,z_n,k)\ \pi_{\lambda,\sigma}(z_1,z_2,\ldots,z_n,k)^\ast g(u)}\ dz_1\ dz_2\ \ldots\ dz_n\ dk\right] &\\[4pt]
&=\lim_{h\rightarrow 0}\left[\dfrac{e^{-i \lambda h u_{11}}-1}{h}\right]\dint_{\R^n}\dint_{K}{f(z_1,z_2,\ldots,z_n,k)\ \pi_{\lambda,\sigma}(z_1,z_2,\ldots,z_n,k)^\ast g(u)}\ dz_1\ dz_2\ \ldots\ dz_n\ dk &\\[4pt]
&=i\lambda u_{11}\dint_{\R^n}\dint_{K}{f(z_1,z_2,\ldots,z_n,k)\ \pi_{\lambda,\sigma}(z_1,z_2,\ldots,z_n,k)^\ast g(u)}\ dz_1\ dz_2\ \ldots\ dz_n\ dk &\\*
&=i\lambda u_{11}\ \widehat{f}(\lambda,\sigma)g(u). &
\end{flalign*}
Hence,
\begin{flalign*}
\left\|\left(\dfrac{\partial f}{\partial z_1}\right)^{\widehat{\ }}(\lambda,\sigma)\right\|_{\text{HS}}^2 &=\sum_{m=0}^{\infty}\sum_{j=1}^{d_m}\dint_K{|i\lambda u_{11}\ \widehat{f}(\lambda,\sigma)g_{mj}(u)|^2}\ du &\\
&\leq \lambda^2\sum_{m=0}^{\infty}\sum_{j=1}^{d_m}\dint_K{|\widehat{f}(\lambda,\sigma)g_{mj}(u)|^2}\ du = \lambda^2\|\widehat{f}(\lambda,\sigma)\|_{\text{HS}}^2. &
\end{flalign*}
So, \eqref{step7-Mn} can be written as
\begin{flalign}
\dfrac{\|f\|_2^2}{2\sqrt{c_n}}&\leq\left(\dint_{K}\dint_{\R^n}{\|z\|^{2a}\ |f(z,k)|^2}\ dz\ dk\right)^{\frac{1}{2a}}\left(\|f\|_2^2\right)^{\frac{1}{2}-\frac{1}{2a}}\times \nonumber &\\*
&\qquad \left(\dint_{0}^{\infty}\sum_{\sigma\in \widehat{M}}{d_\sigma\ \lambda^2\|\widehat{f}(\lambda,\sigma)\|_{\text{HS}}^2}\ \lambda^{n-1}\ d\lambda\right)^{1/2}. \label{step8-Mn} &
\end{flalign}
Now, again using H\"older's inequality, we have
\begin{flalign*}
&\left(\dint_{0}^{\infty}\sum_{\sigma\in \widehat{M}}{d_\sigma\ \lambda^{2b}\|\widehat{f}(\lambda,\sigma)\|_{\text{HS}}^2}\ \lambda^{n-1}\ d\lambda\right)^{\frac{1}{b}}\left(\dint_{0}^{\infty}\sum_{\sigma\in \widehat{M}}{d_\sigma\|\widehat{f}(\lambda,\sigma)\|_{\text{HS}}^2}\ \lambda^{n-1}\ d\lambda\right)^{1-\frac{1}{b}}&\\
&\geq \dint_{0}^{\infty}\sum_{\sigma\in \widehat{M}}{d_\sigma^{1/b}\ \lambda^2\|\widehat{f}(\lambda,\sigma)\|_{\text{HS}}^{\frac{2}{b}}\ d_\sigma^{\left(1-\frac{1}{b}\right)}\|\widehat{f}(\lambda,\sigma)\|_{\text{HS}}^{2\left(1-\frac{1}{b}\right)}}\ \lambda^{n-1}\ d\lambda & \\
&= \dint_{0}^{\infty}\sum_{\sigma\in \widehat{M}}{d_\sigma\ \lambda^2\ \|\widehat{f}(\lambda,\sigma)\|_{\text{HS}}^2}\ \lambda^{n-1}\ d\lambda, &
\end{flalign*}
which implies
\begin{flalign}
&\dint_{0}^{\infty}\sum_{\sigma\in \widehat{M}}{d_\sigma\ \lambda^2\ \|\widehat{f}(\lambda,\sigma)\|_{\text{HS}}^2}\ \lambda^{n-1}\ d\lambda& \leq \left(\dint_{0}^{\infty}\sum_{\sigma\in \widehat{M}}{d_\sigma\ \lambda^{2b}\|\widehat{f}(\lambda,\sigma)\|_{\text{HS}}^2}\ \lambda^{n-1}\ d\lambda\right)^{\frac{1}{b}}\left(\|f\|_2^2\right)^{1-\frac{1}{b}}.   \label{step9-Mn}&
\end{flalign}
Combining \eqref{step8-Mn} and \eqref{step9-Mn}, we obtain
\begin{flalign*}
\dfrac{\|f\|_2^{\left(\frac{1}{a}+\frac{1}{b}\right)}}{2\sqrt{c_n}} &\leq \left(\dint_{K}\dint_{\R^n}{\|z\|^{2a}\ |f(z,k)|^2}\ dz\ dk\right)^{\frac{1}{2a}}\left(\dint_{0}^{\infty}\sum_{\sigma\in \widehat{M}}{d_\sigma\ \lambda^{2b}\|\widehat{f}(\lambda,\sigma)\|_{\text{HS}}^2}\ \lambda^{n-1}\ d\lambda\right)^{\frac{1}{2b}}.& 
\end{flalign*}
\end{proof}
\section{A Class of Nilpotent Lie Groups}\label{ACNLG}
In this section, we shall prove Heisenberg uncertainty inequality for a class of connected, simply connected nilpotent Lie groups $G$ for which the Hilbert-Schmidt norm of the group Fourier transform $\pi_\xi(f)$ of $f$ attains a particular form. \\ 
Let $\g$ be an $n$-dimensional real nilpotent Lie algebra, and let $G=\expo \g$ be the associated connected and simply connected nilpotent Lie group \cite{Cor:Gre:90}. Let $\B=\{X_1,X_2,\ldots,X_n\}$ be a strong Malcev basis of $\g$ through the ascending central series of $\g$. We introduce a `norm function' on $G$ by setting, for $x=\expo(x_1X_1+x_2X_2+\ldots+x_nX_n) \in G$, $x_j\in \R$,
\begin{flalign*}
&&\|x\|&=(x_1^2+\ldots+x_n^2)^{1/2}.&
\end{flalign*}
The composed map
\begin{flalign*}
&&&\R^n\rightarrow \g \rightarrow G,&
\end{flalign*}
given as
\begin{flalign*}
&&&(x_1,\ldots,x_n)\rightarrow \sum_{j=1}^{n}{x_jX_j} \rightarrow \expo\left(\sum_{j=1}^{n}{x_jX_j}\right),&
\end{flalign*}
is a diffeomorphism and maps Lebesgue measure on $\R^n$ to Haar measure on $G$. In this manner, we shall always identify $\g$, and sometimes $G$, as sets with $\R^n$. Thus, measurable (integrable) functions on $G$ can be viewed as such functions on $\R^n$.

Let $\g^\ast$ denote the vector space dual of $\g$ and $\{X_1^\ast,\ldots,X_n^\ast\}$ the basis of $\g^\ast$ which is dual to $\{X_1,\ldots,X_n\}$. Then, $\{X_1^\ast,\ldots,X_n^\ast\}$ is a Jordan-H\"older basis for the coadjoint action of $G$ on $\g^\ast$. We shall identify $\g^\ast$ with $\R^n$ via the map 
\begin{flalign*}
&&&\xi=(\xi_1,\ldots,\xi_n)\rightarrow \sum_{j=1}^{n}{\xi_jX_j^\ast}&
\end{flalign*}
and on $\g^\ast$ we introduce the Euclidean norm relative to the basis $\{X_1^\ast,\ldots,X_n^\ast\}$, i.e.
\begin{flalign*}
&&&\left\|\sum_{j=1}^{n}{\xi_jX_j^\ast}\right\| =(\xi_1^2+\ldots+\xi_n^2)=\|\xi\|.&
\end{flalign*}
Let $\g_j=\R$-span$\{X_1,\ldots,X_n\}$. For $\xi \in \g^\ast$, $\OO_\xi$ denotes the coadjoint orbit of $\xi$. An index $j \in \{1,2,\ldots,n\}$ is a jump index for $\xi$ if
\begin{flalign*}
&&&\g(\xi)+\g_j \neq \g(\xi)+\g_{j-1}.&
\end{flalign*}
We consider,
\begin{flalign*}
&&&e(\xi)=\{j: j\ \text{is a jump index for}\ \xi\}.&
\end{flalign*}
This set contains exactly $\dim(\OO_l)$ indices. Also, there are two disjoint sets $S$ and $T$ of indices with $S \cup T =\{1,\ldots,n\}$ and a $G$-invariant Zariski open set $\U$ of $\g^\ast$ such that $e(\xi)=S$ for all $\xi \in \U$. We define the Pfaffian $\pf(\xi)$ of the skew-symmetric matrix $M_S(\xi)=(\xi([X_i,X_j]))_{i,j\in S}$ as,
\begin{flalign*}
&&&|\pf(\xi)|^2=\det{M_S(\xi)}. &
\end{flalign*}
Let $V_S=\R$-span$\{X_i^\ast: i \in S\}$, $V_T=\R$-span$\{X_i^\ast: i \in T\}$ and $d\xi$ be the Lebesgue measure on $V_T$ such that the unit cube spanned by $\{X_i^\ast:i \in T\}$ has volume $1$. Then, $\g^\ast =V_T \oplus V_S$ and $V_T$ meets $\U$. Let $\W=\U \cap V_T$ be the cross-section for the coadjoint orbits through the points in $\U$. If $d\xi$ is the Lebesgue measure on $\W$, then $d\mu(\xi)=|\pf(\xi)|\ d\xi$ is a Plancherel measure for $\hat{G}$. The Plancherel formula is given by,
\begin{flalign*}
&&&\|f\|_2^2=\dint_{\W}{\|\pi_\xi{(f)}\|_{\text{HS}}^2}\ d\mu(\xi),\quad f \in L^1\cap L^2(G),&
\end{flalign*}
where $\|\pi_\xi{(f)}\|_{\text{HS}}$ denotes the Hilbert-Schmidt norm of $\pi_\xi{(f)}$ and $dg$ is the Haar measure on $G$.\\
We shall consider the case in which $\W$ takes the following form:
\begin{flalign*}
&&\W&=\{\xi=(\xi_1,\xi_2,\ldots,\xi_n)\in \g^\ast:\ \xi_j=0\ \text{for $(n-k)$ values of $j$ with}\ |\pf(\xi)|\neq 0\}.&
\end{flalign*}
We denote the vanishing variables by $\xi_{j_1},\xi_{j_2},\ldots,\xi_{j_{n-k}}$.\\
We consider the class of groups for which for all $\xi\in \W$ and $f \in L^2(G)$ the Hilbert-Schmidt norm $\|\pi_\xi(f)\|_{\text{HS}}^2$ has the following form:
\begin{flalign*}
&&\|\pi_\xi(f)\|_{\text{HS}}^2&=|h(\xi)|\dint_{\R^{n-k}}{\left|\F(f\circ \expo)\left(\xi_1,\xi_2+Q_2,\ldots,\xi_n+Q_n\right)\right|^2}\ d\xi_{j_1}\ d\xi_{j_2}\ \ldots\ d\xi_{j_{n-k}}, &
\end{flalign*}
where $\F$ denotes the Fourier transform on $\R^{n-k}$; $h$ is a function from $\W$ to $\R$ which is non-zero on $\W$ and the functions $Q_m=Q_m(\xi_1,\xi_2,\ldots,\xi_{m-1})$ with $2\leq m \leq n$.

We have the following Heisenberg uncertainty inequality for such groups:
\begin{thm}
For any $f \in L^1\cap L^2(G)$ and $a,b\geq 1$, we have 
\begin{equation}
\dfrac{\|f\|_2^{\left(\frac{1}{a}+\frac{1}{b}\right)}}{4\pi} \leq \left(\dint_{G}{\|x\|^{2a}\ |f(x)|^2}\ dx\right)^{\frac{1}{2a}}\left(\dint_{\W}{\|\xi\|^{2b}\ \|\pi_\xi(f)\|_{\text{HS}}^2}\ \dfrac{1}{|h(\xi)|^b|\pf(\xi)|^{b-1}}\ d\xi\right)^{\frac{1}{2b}}. \label{heq-ACNLG}
\end{equation}
\end{thm}
\begin{proof}
Assuming both the integrals on right hand side of \eqref{heq-ACNLG} finite, we have
\begin{flalign}
&\left(\dint_{G}{\|x\|^2\ |f(x)|^2}\ dx\right)^{1/2}\left(\dint_{\W}{\|\xi\|^2\ \|\pi_\xi(f)\|_{\text{HS}}^2}\ \dfrac{1}{|h(\xi)|}\ d\xi\right)^{1/2} \nonumber&\\*
&=\left(\dint_{\R^n}{\sum_{i=1}^{n}{|x_i|^2}\ \left|(f\circ \exp)\left(\sum_{i=1}^{n}{x_i X_i}\right)\right|^2}\ dx_1\ \ldots\ dx_n\right)^{1/2}\times \nonumber&\\
&\qquad \left(\dint_{\R^k}\dint_{\R^{n-k}}{\sum_{i=1}^{n}{|\xi_i|^2}\ \left|\F(f\circ \expo)\left(\xi_1,\xi_2+Q_2,\ldots,\xi_n+Q_n\right)\right|^2}\ d\xi_1\ \ldots\ d\xi_n\right)^{1/2}  \nonumber&\\[10pt]
&\geq \left(\dint_{\R^n}{|x_1|^2\ \left|(f\circ \exp)\left(\sum_{i=1}^{n}{x_i X_i}\right)\right|^2}\ dx_1\ \ldots\ dx_n\right)^{1/2} \times \nonumber&\\*
&\qquad \left(\dint_{\R^k}\dint_{\R^{n-k}}{|\xi_1|^2\ \left|\F(f\circ \expo)\left(\xi_1,\xi_2+Q_2,\ldots,\xi_n+Q_n\right)\right|^2}\ d\xi_1\ \ldots\ d\xi_n\right)^{1/2}.  \nonumber&\\
&=\left(\dint_{\R^n}{|x_1|^2\ \left|F\left(x_1,\ldots,x_n\right)\right|^2}\ dx_1\ \ldots\ dx_n\right)^{1/2}\times \nonumber&\\*
&\qquad \left(\dint_{\R^n}{|\xi_1|^2\ \left|\widehat{F}\left(\xi_1,\xi_2,\ldots,\xi_n\right)\right|^2}\ d\xi_1\ d\xi_2 \ \ldots\ d\xi_n\right)^{1/2},  & \label{step0-ACNLG}
\end{flalign}
where $F(x_1,\ldots,x_n)=(f\circ \exp)\left(\sum_{i=1}^{n}{x_i X_i}\right)$ which is in $L^2(R^n)$ and $\hat{F}$ being its Fourier transform. \\
By Heisenberg inequality on $\R^n$, we have
\begin{flalign}
\dfrac{\|F\|_2^2}{4\pi} &\leq \left(\dint_{\R^n}{|x_1|^2\ \left|F\left(x_1,\ldots,x_n\right)\right|^2}\ dx_1\ \ldots\ dx_n\right)^{1/2}\times \nonumber&\\
&\qquad \left(\dint_{\R^n}{|\xi_1|^2\ \left|\widehat{F}\left(\xi_1,\xi_2,\ldots,\xi_n\right)\right|^2}\ d\xi_1\ d\xi_2 \ \ldots\ d\xi_n\right)^{1/2}.  & \label{step00-ACNLG}
\end{flalign}
\vspace{-15pt}
\begin{flalign}
\text{But,}\ \|F\|_2^2&=\dint_{\R^n}{|F(x_1,\ldots,x_n)|^2}\ dx_1\ \ldots\ dx_n \nonumber&\\
&=\dint_{\R^n}{\left|(f\circ \exp)\left(\sum_{i=1}^{n}{x_i X_i}\right)\right|^2}\ dx_1\ \ldots\ dx_n=\dint_{G}{|f(x)|^2}\ dx =\|f\|_2^2. \hspace{-20pt}&\label{step000-ACNLG}
\end{flalign}
Combining \eqref{step0-ACNLG}, \eqref{step00-ACNLG} and \eqref{step000-ACNLG}, we get
\begin{flalign}
&&\dfrac{\|f\|_2^2}{4\pi} &\leq \left(\dint_{G}{\|x\|^2\ |f(x)|^2}\ dx\right)^{1/2}\left(\dint_{\W}{\|\xi\|^2\ \|\pi_\xi(f)\|_{\text{HS}}^2}\ \dfrac{1}{|h(\xi)|}\ d\xi\right)^{1/2}.  \label{step1-ACNLG}&
\end{flalign}
Now, as in the proof of Theorem \ref{theq-Mn}, applications of H\"older's inequality give
\begin{flalign}
&&\dint_{G}{\|x\|^2\ |f(x)|^2}\ dx \leq &\left(\dint_{G}{\|x\|^{2a}\ |f(x)|^2}\ dx\right)^{\frac{1}{a}}\left(\|f\|_2^2\right)^{1-\frac{1}{a}}  \label{step2-ACNLG}&
\end{flalign}
and
\begin{flalign}
\dint_{\W}{\|\xi\|^2\ \|\pi_\xi(f)\|_{\text{HS}}^2}\ \dfrac{1}{|h(\xi)|}\ d\xi \leq &\left(\dint_{\W}{\|\xi\|^{2b}\ \|\pi_\xi(f)\|_{\text{HS}}^2}\ \dfrac{1}{|h(\xi)|^b|\pf(\xi)|^{b-1}}\ d\xi\right)^{\frac{1}{b}}\left(\|f\|_2^2\right)^{1-\frac{1}{b}}.  \label{step3-ACNLG}&
\end{flalign}
Combining \eqref{step1-ACNLG}, \eqref{step2-ACNLG} and \eqref{step3-ACNLG}, we obtain
\begin{flalign*}
\dfrac{\|f\|_2^{\left(\frac{1}{a}+\frac{1}{b}\right)}}{4\pi} &\leq \left(\dint_{G}{\|x\|^{2a}\ |f(x)|^2}\ dx\right)^{\frac{1}{2a}}\left(\dint_{\W}{\|\xi\|^{2b}\ \|\pi_\xi(f)\|_{\text{HS}}^2}\ \dfrac{1}{|h(\xi)|^b|\pf(\xi)|^{b-1}}\ d\xi\right)^{\frac{1}{2b}}.&
\end{flalign*}
\vspace{-10pt}
\end{proof}
\begin{ex}
We now list several classes that are included in the above general class. 
\renewcommand{\labelenumi}{\arabic{enumi}.}
\begin{enumerate}
\item For thread-like nilpotent Lie groups \text{(}for details, see \cite{Kan:Kum:01}\text{)}, we have $\pf(\xi)=\xi_1$ and
\begin{flalign*}
&&\W&=\{\xi=(\xi_1,0,\xi_3,\ldots,\xi_{n-1},0):\xi_j \in \R, \xi_1\neq 0\}.&
\end{flalign*}
Also, $\|\pi_\xi(f)\|_{\text{HS}}$ is given by
\begin{flalign*}
\|\pi_\xi(f)\|_{\text{HS}}^2&=\dfrac{1}{|\xi_1|}\dint_{\R^2}{\left|\F{(f\circ \expo)}\left(\xi_1,t,\xi_3+Q_3,\ldots,\xi_{n-1}+Q_{n-1},s\right)\right|^2}\ ds\ dt, &
\end{flalign*}
where $Q_j(\xi_1,0,\xi_3,\ldots,\xi_{j-1},t)=\sum\limits_{k=1}^{j-1}{\dfrac{1}{k!}\ \dfrac{t^k}{\xi_1^k}\ \xi_{j-k}}$, for $3\leq j\leq n-1$.\\
Thus, for $h(\xi)=\dfrac{1}{|\xi_1|}=\dfrac{1}{|\pf(\xi)|}$, one obtains the Heisenberg uncertainty inequality,
\begin{flalign*}
&&\dfrac{\|f\|_2^{\left(\frac{1}{a}+\frac{1}{b}\right)}}{4\pi} &\leq \left(\dint_{G}{\|x\|^{2a}\ |f(x)|^2}\ dx\right)^{\frac{1}{2a}}\left(\dint_{\W}{\|\xi\|^{2b}\ \|\pi_\xi(f)\|_{\text{HS}}^2}\ |\xi_1|\ d\xi\right)^{\frac{1}{2b}}.&
\end{flalign*}
\item For $2$-NPC nilpotent Lie groups \text{(}for details, see \cite{Bak:Sal:08}\text{)}, let $\{0\}=\g_0 \subset \g_1 \subset \cdots \subset \g_n =\g$ be a Jordan-H\"older sequence in $\g$ such that $\g_m=\z(g)$ and $\hf=\g_{n-2}$. Let us consider the ideal $[\g,\g_{m+1}]$ of $\g$ which is one or two dimensional in $\g$. We discuss the two cases separately:\\
(a) $\dim{[\g,\g_{m+1}]}=2$. \\
In this case, for every basis $\{X_1,X_2\}$ of $\hf$ in $\g$ and every $Y_1 \in \g_{m+1}\setminus\z(\g)$, the vectors $Z_1=[X_1,Y_1]$ and $Z_2=[X_2,Y_1]$ are linearly independent and lie in the center of $\g$. Assume that $\g_1=\R$-span$\{Z_1\}$, $\g_2=\R$-span$\{Z_1,Z_2\}$. Let $Z_3,\ldots,Z_m$ be some vectors such that $\z(\g)=\R$-span$\{Z_1,\ldots,Z_m\}$ and $\B=\{Z_1,\ldots,Z_n\}$ a Jordan-H\"older basis of $\g$ chosen as follows:
\renewcommand{\labelenumii}{(\roman{enumii})}
\begin{enumerate}
\item $\z(\g)=\R$-span$\{Z_1,\ldots,Z_m\}$
\item $\hf=\R$-span$\{Z_1,\ldots,Z_{n-2}\}$
\item $\g=\R$-span$\{Z_1,\ldots,Z_{n-2}, X_1=Z_{n-1},X_2=Z_n\}$.
\end{enumerate}
For $m_1=m+1$ and $m+2\leq m_2 \leq n-2$, we denote $Z_{m_1}=Z_{m+1}=Y_1$, $Z_{m_2}=Y_2$. These vectors can be chosen such that $\xi_1=\xi([X_1,Y_1])\neq 0$, $\xi_{2,2}=\xi([X_2,Y_2])\neq 0$, for all $\xi\in \W$, where
\begin{flalign*}
\W&=\{\xi=(\xi_1,\xi_2,\ldots,\xi_m,0,0,\xi_{m+3},\xi_{m+4},\ldots,\xi_{n-2},0,0):\xi_j \in \R, |\pf(\xi)|\neq 0\}.&
\end{flalign*}
 Also, we have $\pf(\xi)=\xi(Z_1)\ \xi([X_2,Y_2])-\xi([X_1,Y_2])\ \xi(Z_2)$ and $\|\pi_\xi(f)\|_{\text{HS}}$ is given by,
\begin{flalign*}
\|\pi_\xi(f)\|_{\text{HS}}^2&=|h(\xi)|\dint_{\R^4}{\left|\F{(f\circ \expo)}\left(s_2,s_1,P_{n-2}\left(\xi,-\dfrac{t_1}{\tilde{\xi}_{1,1}},-\dfrac{t_2}{\tilde{\xi}_{2,2}}\right),\ldots,\right.\right.}& \\*
&\qquad{\left.\left. P_{m+3}\left(\xi,-\dfrac{t_1}{\tilde{\xi}_{1,1}},-\dfrac{t_2}{\tilde{\xi}_{2,2}}\right),t_2,t_1,\xi_m,\ldots,\xi_1\right)\right|^2}\ ds_1\ ds_2\ dt_1\ dt_2, &
\end{flalign*}
where $h$ is the function defined by
\begin{flalign*}
&&h(\xi)&=\dfrac{|\xi_1\xi_{2,2}|^2}{|\xi_1\xi_{2,2}-\xi_{1,2}\xi_2|^2},&
\end{flalign*} 
$\xi_{i,j}=\xi([X_i,Y_j])$, $\tilde{\xi}_{i,j}=\xi([X_i(\xi)),Y_j])$ and $P_j(\xi,t)$ is a polynomial function with respect to the variables $t=(t_1,t_2)$ and $\xi_{m+1},\ldots,\xi_j$ and rational in the variables $\xi_1,\ldots,\xi_m$.
Thus, one obtains the Heisenberg uncertainty inequality,
\begin{flalign*}
\dfrac{\|f\|_2^{\left(\frac{1}{a}+\frac{1}{b}\right)}}{4\pi} &\leq \left(\dint_{G}{\|x\|^{2a}\ |f(x)|^2}\ dx\right)^{\frac{1}{2a}}\left(\dint_{\W}{\|\xi\|^{2b}\ \|\pi_\xi(f)\|_{\text{HS}}^2}\ \dfrac{1}{|h(\xi)|^b|\pf(\xi)|^{b-1}}\ d\xi\right)^{\frac{1}{2b}}.&
\end{flalign*}
(b) $\dim{[\g,\g_{m+1}]}=1$. \\
In this case, we have $\pf(\xi)=\xi([X_1,Y_1])\cdot \xi([X_2,Y_2])$ and
\begin{flalign*}
&&\W&=\left\{\xi=(\xi_1,\xi_2,\ldots,\xi_m,0,\xi_{m+2},\ldots,\xi_{m+d+1},0,\xi_{m+d+3},\ldots,\xi_{n-2},0,0)\right.&\\*
&&&\hspace{150pt} \left.:\xi_j \in \R, |\pf(\xi)|\neq 0\right\}.&
\end{flalign*}
Also, $\|\pi_\xi(f)\|_{\text{HS}}$ is given by,
\begin{flalign*}
\|\pi_\xi(f)\|_{\text{HS}}^2&=\dfrac{1}{|\pf(\xi)|}\dint_{\R^4}{\left|\F{(f\circ \expo)}\left(s_2,s_1,P_{n-2}\left(\xi,-\dfrac{t_1}{\xi_1},-\dfrac{t_2+R(-\frac{t_1}{\xi_1},\xi_1,\ldots,\xi_{m+d})}{\xi_{2,2}}\right),\right.\right.}& \\
&\qquad\qquad\qquad {\left.\left. \ldots,t_2,\ldots,P_{m+2}\left(\xi,-\dfrac{t_1}{\xi_1}\right),t_1,\xi_m,\ldots,\xi_1\right)\right|^2}\ ds_1\ ds_2\ dt_1\ dt_2. &
\end{flalign*}
Thus, for $h(\xi)=\dfrac{1}{|\pf(\xi)|}$, one obtains the Heisenberg uncertainty inequality,
\begin{flalign*}
\dfrac{\|f\|_2^{\left(\frac{1}{a}+\frac{1}{b}\right)}}{4\pi}& \leq \left(\dint_{G}{\|x\|^{2a}\ |f(x)|^2}\ dx\right)^{\frac{1}{2a}}\left(\dint_{\W}{\|\xi\|^{2b}\ \|\pi_\xi(f)\|_{\text{HS}}^2}\ |\pf(\xi)|\ d\xi\right)^{\frac{1}{2b}}.&
\end{flalign*}
\item For connected, simply connected nilpotent Lie groups $G=\expo{\g}$ such that $\g(\xi) \subset [\g,\g]$ for all $\xi \in \U$ \text{(}for details, see \cite{Sma:11}\text{)}, we consider $S=\{j_1<\ldots<j_d\}$ and $T=\{t_1<\ldots<t_r\}$ to be the collection of jump and non-jump indices respectively, with respect to the basis $\B$. We have, $j_d=n$ and
\begin{flalign*}
&&\W&=\{\xi=(\xi_1,\xi_2,\ldots,\xi_n)\in \g^\ast: \xi_{j_i}=0\ \text{for $j_i \in S$ with}\ |\pf(\xi)|\neq 0\}.&
\end{flalign*} 
Also, $\|\pi_\xi(f)\|_{\text{HS}}$ is given by,
\begin{flalign*}
&&\|\pi_\xi(f)\|_{\text{HS}}^2&=\dfrac{|\xi([X_{j_1},X_n])|}{|\pf(\xi)|^2}\dint_{\W}{\left|\F{(f\circ \expo)}\left(\xi,w\right)\right|^2}\ dw, &
\end{flalign*}
where $\xi=(\xi_{t_i})_{t_i\in T}$ and $w=(w_{j_i})_{j_i\in S}$. Thus, for $h(\xi)=\dfrac{|\xi([X_{j_1},X_n])|}{|\pf(\xi)|^2}$, one obtains the Heisenberg uncertainty inequality,
\begin{flalign*}
\dfrac{\|f\|_2^{\left(\frac{1}{a}+\frac{1}{b}\right)}}{4\pi} &\leq \left(\dint_{G}{\|x\|^{2a}\ |f(x)|^2}\ dx\right)^{\frac{1}{2a}}\left(\dint_{\W}{\|\xi\|^{2b}\ \|\pi_\xi(f)\|_{\text{HS}}^2}\ \dfrac{|\pf(\xi)|^{b+1}}{|\xi([X_{j_1},X_n])|^b}\ d\xi\right)^{\frac{1}{2b}}.&
\end{flalign*}
\item For low-dimensional nilpotent Lie groups of dimension less than or equal to $6$ \text{(}for details, see \cite{Nie:83}\text{)} except for $G_{6,8}$, $G_{6,12}$, $G_{6,14}$, $G_{6,15}$, $G_{6,17}$, an explicit form of $\|\pi_\xi(f)\|_{\text{HS}}$ can be obtained. Thus, an explicit Heisenberg uncertainty inequality can be written down.
\item The classes mentioned above are distinct. For instance, $G_{5,5}$ is thread-like nilpotent Lie group, but it does not belong to class mentioned in Example 3. Also, $G_{5,3}$ belongs to class mentioned in Example 3, but it is not thread-like nilpotent Lie group.
\end{enumerate}
\end{ex}
\section*{Competing Interests}
The authors declare that they have no competing interests.
\section*{Authors' Contributions}
All authors contributed equally to this paper and they read and approved the final manuscript.
\section*{Acknowledgements}
The second author was supported by R \& D grant of University of Delhi. The authors would like to thank the referees for many valuable suggestions which helped in improving the exposition.

\end{document}